\documentclass{article}

\usepackage{amssymb}
\usepackage{latexsym}
\usepackage{extarrows}
\usepackage{float}
\usepackage[dvips]{graphicx}

\newtheorem{theo}{Theorem}[section]

\newtheorem{lemma}[theo]{Lemma}

\newtheorem{con}[theo]{Conjecture}
\newtheorem{prop}[theo]{Proposition}
\newtheorem{fact}[theo]{Fact}
\newtheorem{defi}[theo]{Definition}
\newtheorem{algorithm}[theo]{Algorithem}

\newcommand{\qed}{\hspace*{\fill} \rule{7pt}{7pt}}

\def\endproofbox{\hskip 1.3em\hfill\rule{6pt}{6pt}}

\newenvironment{proof}%
{%
\noindent{\it Proof.}
}%
{%
 \quad\hfill\endproofbox\vspace*{2ex}
}

\def\qed{\hskip 1.3em\hfill\rule{6pt}{6pt}}

\begin{document}
\title{Lagrangian densities of short $3$-uniform linear paths and Tur\'an numbers of their extensions}
    \author{Biao Wu \thanks{MOE-LCSM, School of Mathematics and Statistics, Hunan Normal University, Changsha, Hunan 410081, P. R. China. Email: wu@hunnu.edu.cn. \ Supported by the Construct Program of the Key Discipline in Hunan Province.}  \and  Yuejian Peng \thanks{Corresponding author. Institute of Mathematics, Hunan University, Changsha, 410082, P.R. China. Email: ypeng1@hnu.edu.cn. \ Supported in part by National Natural Science Foundation of China (No. 11671124).}
}

\date{}
\maketitle
\begin{abstract}
For a fixed positive integer $n$ and an $r$-uniform hypergraph $H$, the Tur\'an number $ex(n,H)$ is the maximum number of edges in an $H$-free $r$-uniform hypergraph on $n$ vertices,
and the Lagrangian density  of $H$ is defined as
$\pi_{\lambda}(H)=\sup \{r! \lambda(G) : G \;\text{is an}\; H\text{-free} \;r\text{-uniform hypergraph}\}$, where $\lambda(G)=\max\{\sum_{e \in G}\prod\limits_{i\in e}x_{i}: x_i\ge 0 \;\text{and}\; \sum_{i\in V(G£©} x_i=1\}$ is the Lagrangian of $G$.
For an $r$-uniform hypergraph $H$ on $t$ vertices, it is clear that $\pi_{\lambda}(H)\ge r!\lambda{(K_{t-1}^r)}$. Let us say that an $r$-uniform hypergraph $H$ on $t$ vertices is perfect if $\pi_{\lambda}(H)= r!\lambda{(K_{t-1}^r)}$. A result of Motzkin and Straus imply that all graphs are perfect. It is interesting to explore what kind of hypergraphs are perfect.
Let $P_t=\{e_1, e_2, \dots, e_t\}$ be the linear $3$-uniform path of length $t$, that is, $|e_i|=3$, $|e_i \cap e_{i+1}|=1$ and $e_i \cap e_j=\emptyset$ if $|i-j|\ge 2$.
We show that $P_3$ and $P_4$ are perfect, this supports a conjecture in \cite{yanpeng} proposing that all $3$-uniform linear hypergraphs are perfect.
Applying the results on Lagrangian densities, we determine the Tur\'an numbers of their extensions.
\end{abstract}


\section{Introduction}
For a positive integer $n$, let $[n]$ denote $\{1, 2, \ldots, n\}$.
 An {\em $r$-uniform hypergraph} or {\em $r$-graph $G$} consists of a set $V(G)$ of vertices and a set $E(G) \subseteq V(G) ^{(r)}$ of edges.
A $2$-graph is called a simple graph.
We write $G$ for $E(G)$ sometimes.
 An edge $e=\{a_1, a_2, \ldots, a_r\}$ will be simply denoted by $a_1a_2 \ldots a_r$.
 An $r$-graph $F$ is  a {\it subgraph} of an $r$-graph $G$, denoted by $F\subseteq G$, if $V(F)\subseteq V(G)$ and $E(F)\subseteq E(G)$.
 Given an $r$-graph $G$ and $U \subseteq V(G)$, the {\em induced subgraph} $G[U]$ is the $r$-graph with vertex set $U$ and edge set $\{e\in G: e\subseteq U\}$.
 Let $K^{r}_t$ denote the {\em complete $r$-graph} on $t$ vertices, and $K^{r-}_t$ be removing one edge from  $K^{r}_t$.
A hypergprah $H$ {\em covers pairs} if every pair of vertices is contained in some edge of $H$.
 The {\em extension} of an $r$-graph $F$, denoted by $H^F$, is defined as follows.
 For each pair of vertices $v_i,v_j \in V(F)$ not covered in $F$, we add a set $B_{ij}$ of $r-2$ new vertices and the edge $\{v_i,v_j\} \cup B_{ij}$, where all $B_{ij}$ are pairwise disjoint over all such pairs $\{i,j\}$.

Given an $r$-graph $F$,  an $r$-graph $G$ is called \emph{$F$-free} if it does not contain a copy of $F$ as a subgraph.
For a fixed positive integer $n$ and an $r$-graph $F$, the {\em Tur\'an number} of $F$, denoted by $ex(n,F)$, is the maximum number of edges in an $F$-free $r$-graph on $n$ vertices.
Determining the value $ex(n,F)$ for a general $r$-graph $F$ is a challenging problem in extremal combinatorics.
For simple graphs, Erd\H{o}s, Stone and Simonovits determined the asymptotic value of Tur\'an numbers of all graphs except bipartite graphs. Very few results are known for hypergraphs and a survey on this topic can be found in Keevash's survey paper \cite{Keevash}.
Lagrangian method has been a helpful tool for hypergraph Tur\'an problem. We now proceed to define the Lagrangian of an $r$-graph.

\begin{defi}
Let $G$ be an $r$-graph on $[n]$ and let
  $\vec{x}=(x_1,\ldots,x_n) \in [0,\infty)^n$,
define
$$\lambda (G,\vec{x})=\sum_{e \in G}\prod\limits_{i\in e}x_{i}.$$
\end{defi}
Denote $$\Delta_n=\{\vec{x} =(x_1,x_2,\ldots ,x_n) \in [0,\infty)^{n}: x_1+x_2+\dots+ x_n =1\}.$$
The {\em Lagrangian} of
$G$, denoted by $\lambda (G)$, is defined as
 $$\lambda (G) = \max \{\lambda (G, \vec{x}): \vec{x} \in \Delta_n \}.$$

The value $x_i$ is called the {\em weight} of the vertex $i$ and a vector $\vec{x} \in {\Delta_n}$ is called a {\em feasible weight vector} on $G$.
A vector  $\vec{y}\in {\Delta_n}$ is called an {\em optimum weight vector} on $G$ if $\lambda (G, \vec{y})=\lambda(G)$.

In \cite{MS}, Motzkin and Straus established a connection between the Lagrangian of any given $2$-graph and it's maximum complete subgraphs.
\begin{theo} {\em(\cite{MS})} \label{MStheo}
If $G$ is a $2$-graph in which a maximum complete subgraph has  $t$ vertices, then
$$\lambda(G)=\lambda(K_t^2)={1 \over 2}(1 - {1 \over t}).$$
\end{theo}

Given an $r$-graph $F$, the {\em Lagrangian density } $\pi_{\lambda}(F)$ of $F$ is defined as
$$\pi_{\lambda}(F)=\sup \{r! \lambda(G): G \;\text{is an}\; F\text{-free} \;r\text{-graph}\}.$$
The Lagrangian density of an $r$-graph is closely related to its Tur\'an density.
\begin{prop}{\em (\cite{Sidorenko-87, Pikhurko})}\label{relationlt} Let $F$ be an $r$-graph. Then \\
$(i)$ $\pi(F)\le \pi_{\lambda}(F);$ \\
$(ii)$ $\pi(H^F)=\pi_{\lambda}(F).$ In particular, if $F$ covers pairs, then $\pi(F)= \pi_{\lambda}(F).$
\end{prop}
Earlier applications of Lagrangians of  hypergraphs include that Frankl and R\"{o}dl \cite{FR} applied it in disproving the
long standing jumping constant conjecture of Erd\H{o}s.
Sidorenko \cite{Sidorenko-89},  and Frankl and F\"uredi \cite{FF} applied Lagrangians of hypergraphs in
finding Tur\'an densities of hypergraphs, generalizing work of Motzkin and Straus \cite{MS}, and
Zykov \cite{Z}.
More recent developments of the method were obtained by Pikhurko \cite{Pikhurko} and in the papers \cite{HK, NY, BIJ, NY2, Jenssen}.
 In addition to its applications, it is interesting in its own right to determine the  maximum Lagrangian of $r$-graphs with certain properties. For example, a challenging conjecture of Frankl and F\"uredi \cite{FF} considers the question of determining the maximum  Lagrangian among all $r$-graphs with the fixed  number of edges.  Talbot \cite{T} made some breakthrough in  confirming  this conjecture for some cases.
 Subsequent progress in this conjecture  were made in the papers of Tang, Peng, Zhang and Zhao \cite{TPZZ2}, Tyomkyn \cite{Tyo}, and Lei, Lu and Peng \cite{LLP2018}. Recently, Gruslys, Letzter and Morrison \cite{GLM2018} confirmed this conjecture for $r=3$ and the number of edges is sufficiently large.
 In this paper, we focus on  the Lagrangian density of an $r$-graph $F$.

 For an $r$-graph $H$ on $t$ vertices, it is clear that $\pi_{\lambda}(H)\ge r!\lambda{(K_{t-1}^r)}$. Let us say that an $r$-graph $H$ on $t$ vertices is {\em perfect} if $\pi_{\lambda}(H)= r!\lambda{(K_{t-1}^r)}$.
Theorem \ref{MStheo} implies that all  $2$-graphs are perfect.  It is interesting to explore what kind of hypergraphs are perfect.  
Sidorenko \cite{Sidorenko-89} showed that the $(r-2)$-fold enlargement of a tree with order greater than some number $A_r$ is perfect.
 Hefetz and Keevash \cite{HK} showed that  a $3$-uniform matching  of size 2 is perfect.    Jiang, Peng and Wu   \cite{JPW}  verified  that any $3$-uniform matching is perfect.
 Pikhurko \cite{Pikhurko},  and Norin and Yepremyan  \cite{NY2} showed that an $r$-uniform tight  path of length 2 is perfect for $r=4$ and $r=5$ or $6$ respectively.
 Jenssen \cite{Jenssen} showed that  a path of length 2 formed by two edges intersecting at   $r-2$ vertices is perfect for $r=3, 4, 5, 6, 7$.   An $r$-graph is {\em linear} if any two edges have at most 1 vertex in common.
 Hu, Peng and Wu \cite{HPW},  and Chen, Liang and Peng \cite{CLP} showed that  the disjoint union of a $3$-uniform linear path of length $2$ or $3$ and a  $3$-uniform matching, and the disjoint union of a $3$-uniform tight path of length $2$  and a  $3$-uniform matching are perfect. Yan and Peng \cite{yanpeng} showed that the $3$-uniform linear cycle of length 3 (\{123, 345, 561\}) is perfect, and $F_5$ (\{123, 124, 345\}) is not perfect (by determining its Lagrangian density).
 Bene Watts, Norin and Yepremyan \cite{NWY} showed that  an $r$-uniform matching  of size 2 is not perfect for $r\ge 4$ confirming a conjecture of  Hefetz and Keevash \cite{HK}.
 Wu, Peng and Chen \cite{WPC} showed the same result for $r=4$ independently.
 Though an $r$-uniform matching  of size 2 is not perfect for $r\ge 4$, we think that an $r$-uniform matching with large enough size is perfect. Yan and Peng  proposed  the following conjecture in \cite{yanpeng}.

\begin{con}\label{con1}{\em (\cite{yanpeng})}
 For $r\ge 3$, there exists $n$ such that a linear $r$-graph with at least $n$ vertices is perfect.
\end{con}

 A natural and interesting question is whether a linear hyperpath perfect? Let $P_t=\{e_1, e_2, \dots, e_t\}$ be the $3$-uniform linear path of length $t$, that is, $|e_i|=3$, $|e_i \cap e_{i+1}|=1$ and $e_i \cap e_j=\emptyset$ if $|i-j|\ge 2$. We show that $P_3$ and $P_4$ are perfect in this paper.
In the joint work with Jiang \cite{JPW}, we  applied the fact that left-compressing an $M^r_t$-free $r$-graph yields an $M^r_t$-free $r$-graph, where $M^r_t$ is an $r$-uniform matching of size $t$.
In general, left-compressing a $P_t$-free $3$-graph may not result in a $P_t$-free $3$-graph. However,  we manage to prove that left-compressing a dense $P_t$-free $3$-graph will result in $P_t$-free $3$-graph for $t=3$ or $4$ by structural analysis, and determine the Lagrangian density of $P_3$, and $P_4$.

In the next section, we give some useful properties of the Lagrangian function.
In Section 3, we prove that left-compressing  a $P_3$-free 3-graph ($P_4$-free 3-graph) that covers pairs results in a  $P_3$-free ($P_4$-free) $3$-graph, and
show that $P_t$  is perfect for $t=3$ or $4$.
In Section 4, we give the Tur\'an numbers of their extensions  by using a similar stability argument for lager enough $n$ as in \cite{Pikhurko} and several other papers.


\section{Some properties of the Lagrangian function}

In this section, we develop some useful properties of Lagrangian functions.
The following fact follows immediately from the definition of the Lagrangian.
\begin{fact}\label{mono}
Let $F$, $G$ be $r$-graphs and $F\subseteq G$. Then $\lambda (F) \le \lambda (G).$
\end{fact}

Given an $r$-graph $G$ and a vertex $i \in V(G)$, the {\em link} of $i$ in $G$, denoted by $L_G(i)$, is the $(r-1)$-graph with edge set $\left\{e\in {V(G)\setminus \{i\} \choose r-1}: e \cup \{i\} \in E(G)\right\}$. We will drop the subscript $G$ when there is no confusion.
Given $i,j\in V(G)$, define
$$L_G(j \setminus i)=\left\{f\in \binom{V(G)\setminus \{i,j\}}{r-1}: f \cup \{j\} \in E(G) {\rm \ and \ } f \cup \{i\} \notin E(G)\right\},$$ and define {\em the compression of $j$ to $i$} as
$$\pi_{ij}(G)=\left(E(G)\setminus \{f\cup \{j\}: f\in  L_G(j \setminus i)\} \right) \cup \{ f\cup \{i\}: f \in L_G(j \setminus i) \}.$$
We say $G$ on vertex set $[n]$ is {\em left-compressed} if for every $i,j$, $1\le i < j \le n$, $L_G(j \setminus i)= \emptyset$.
By the definition of $\pi_{ij}(G)$, it's straightforward to verify the following fact.

\begin{fact}\label{compression-preserve}
Let $G$ be an $r$-graph on the vertex set $[n]$. Let $\vec{x}=(x_1,x_2,\dots,x_n)$ be a feasible weight vector on $G$. If $x_i \ge x_j$, then $\lambda(\pi_{ij}(G),\vec{x})\ge \lambda(G,\vec{x})$.
\end{fact}
\medskip

An $r$-graph $G$ is {\em dense} if $\lambda (G') < \lambda (G)$ for every subgraph $G'$ of $G$ with $|V(G')|<|V(G)|$.
 This is equivalent to that  no coordinate in all optimum weight vector is zero.

\begin{fact} {\em (\cite{FR})}\label{dense}
If $G$ is a dense $r$-graph then $G$ covers pairs.
\end{fact}

%

Let $G$ be an $r$-graph on $[n]$ and $\vec{x}=(x_1,x_2,\dots,x_n)$ be a weight vector on $G$.
If we view $\lambda(G,\vec{x})$ as a function in variables $x_1,\dots, x_n$, then
$$ \frac{\partial \lambda (G, \vec{x})}{\partial x_i}=\sum_{i \in e \in E(G)}\prod\limits_{j\in e\setminus \{i\}}x_{j}.$$

\begin{lemma} {\em (\cite{FR})}\label{lemma-partion}
Let $G$ be an $r$-graph on $[n]$. Let $\vec{x}=(x_1,x_2,\dots,x_n)$ be an optimum weight vector on  $G$. Then
$$ \frac{\partial \lambda (G, \vec{x})}{\partial x_i}=r\lambda(G)$$
for every $i \in [n]$ with $x_i>0$.
\end{lemma}

\begin{lemma}\label{Equivalent}
Let $G$ be an $r$-graph on $[n]$. Let $\vec{x}=(x_1,x_2,\dots,x_n)$ be a feasible weight vector on $G$. Let $i,j\in [n]$, where $i\neq j$.
Suppose that $L_G(i \setminus j)=L_G(j \setminus i)=\emptyset$. Let
$\vec{y}=(y_1,y_2,\dots,y_n)$ be defined by letting $y_\ell=x_\ell$ for every $\ell \in [n]\setminus \{i,j\}$ and letting $y_i=y_j=(x_i+x_j)/2$.
Then $\lambda(G,\vec{y})\geq \lambda(G,\vec{x})$. Furthermore, if the pair $\{i,j\}$ is contained in some edge of $G$ and $\lambda(G,\vec{y})=\lambda(G,\vec{x})$, then $x_i=x_j$.
\end{lemma}
\begin{proof}
Since $L_G(i \setminus j)=L_G(j \setminus i)=\emptyset$, we have
$$\lambda(G,\vec{y})-\lambda(G,\vec{x})=\sum_{\{i,j\} \subseteq e \in G}\left[{(x_i+x_j)^2 \over 4}-x_ix_j\right]\prod\limits_{k\in e\setminus \{i,j\}}x_k \ge 0.$$
If the pair $\{i,j\}$ is contained in some edge of $G$, then equality holds only if $x_i=x_j$.
\end{proof}

The following facts are consequences of Lemma \ref{Equivalent}.
\begin{fact}\label{K_t^r}
$\lambda(K_t^{r})={t\choose r}{1\over t^r}.$
\end{fact}

Recall that $K^{r-}_t$ is the $r$-graph by removing one edge from  $K^{r}_t$.

\begin{fact}\label{K_4^3-}
$\lambda(K_4^{3-})={4\over 81}< 0.0494.$
\end{fact}
\begin{proof}
Let $\vec{x}=(x_1,x_2,x_3,x_4)$ be an optimum weighting of $K_4^{3-}$. By Lemma \ref{Equivalent}, we can assume that $x_1=a$ and $x_2=x_3=x_4=b$. So $a+3b=1$.
Then $\lambda(K_4^{3-})= 3ab^2=3(1-3b)b^2 \le {4\over 3}\left(1-3b+1.5b+1.5b\over 2\right)^3= {4\over 81}< 0.0494.$
\end{proof}

\begin{fact}\label{K_6^3-}
$\lambda(K_6^{3-})< 0.0887.$
\end{fact}
\begin{proof}
Let $\vec{x}=(x_1,\dots,x_6)$ be an optimum weighting of $K_6^{3-}$. By Lemma \ref{Equivalent}, we can assume that $x_1=x_2=x_3=a$ and $x_4=x_5=x_6=b$. So $3a+3b=1$.
Then $\lambda(K_6^{3-})= a^3+9a^2b+9ab^2=a^3+3a^2(1-3a)+a(1-3a)^2=a^3-3a^2+a.$
Let $f(a)=a^3-3a^2+a$, we have $f'(a)=3a^2-6a+1$ and $f''(a)=6a-6$.
Since $f''(a)<0$, $f(a)$ reaches the maximaum on interval $[0, 1/3]$ at $a$ satisfying $f'(a)=0$.
Then direct calculation $f(a)\le f({3-\sqrt{6} \over 3})< 0.0887.$
\end{proof}

\begin{fact}\label{K_8^3-}
$\lambda(K_8^{3-})< 0.1077.$
\end{fact}
\begin{proof}
Let $\vec{x}=(x_1,\dots,x_8)$ be an optimum weighting of $K_8^{3-}$. By Lemma \ref{Equivalent}, we can assume that $x_1=\dots=x_5=a$ and $x_6=x_7=x_8=b$. So $5a+3b=1$.
Then $\lambda(K_8^{3-})= 10a^3+30a^2b+15ab^2=(5a^3-20a^2+5a)/3.$

Let $f(a)=(5a^3-20a^2+5a)/3$.
Then $f'(a)=(15a^2-40a+5)/3$ and $f'(a)=0$ implies that $a=(4\pm \sqrt{13})/3$.
It's easy to see that $\max_{0<a<1} f(a)=f((4- \sqrt{13})/3)<0.1077$.
\end{proof}

The following result in \cite{PZ} is useful for determining the Lagrangian of some hypergraph containing a large clique.
\begin{theo}{\rm ( \cite{PZ})} \label{theoPZ}
Let $m$ and $l$ be positive integers satisfying ${l-1 \choose 3} \le m \le {l-1 \choose 3} + {l-2 \choose 2}$. Let $G$ be a $3$-graph with $m$ edges and $G$ contains a complete subgraph of order  $l-1$. Then $\lambda(G) = \lambda([l-1]^{(3)})$.
\end{theo}

\section{The Lagrangian densities of $P_3$ and $P_4$}

\begin{figure}[H]
\centering
\begin{minipage}{7cm}
\includegraphics[width=0.9\textwidth, height=0.08\textwidth]{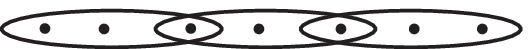}
\label{fig P3}
\caption{$P_3$}
\end{minipage}
\begin{minipage}{7cm}
\includegraphics[width=0.9\textwidth, height=0.08\textwidth]{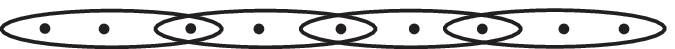}
\label{fig P4}
\caption{$P_4$}
\end{minipage}
\end{figure}

We first study a property of dense $3$-graphs.
\begin{lemma} \label{p_2'}
Let $i=1$ or $2$.
Let $\mathcal{F}$ be a dense $3$-graph with $ n\ge 6-i$ vertices.
Then there are $e_1,e_2 \in \mathcal{F}$ such that $|e_1 \cap e_2|=i$.
\end{lemma}
\begin{proof}
By Fact \ref{dense}, $\mathcal{F}$ being dense implies that $\mathcal{F}$ covers pairs.
For $i=1$, since $n\ge 5$ and $\mathcal{F}$ covers pairs, it is easy to see that $\mathcal{F}$ has at least two edges.
Let $e,f\in \mathcal{F}$.
Let $a\in e\setminus f$ and $b\in f\setminus e$.  Since $\mathcal{F}$ covers pairs, there exists one an edge $g\in \mathcal{F}$ such that $a,b\in g$.
Hence either $|e\cap g|=1$ or $|f\cap g|=1$ (or both).

For $i=2$. Let $\vec{x}=(x_1,x_2,\dots,x_n)$ be an optimum weighting of $\mathcal{F}$. Note that $x_i>0$ for all $i\in [n]$.
Suppose for any edge pair $e_1,e_2 \in \mathcal{F}$, $|e_1 \cap e_2|=0$ or $1$.
This implies that for every pair $i,j\in V(\mathcal{F})$, $\{i,j\}$ is contained in at most one edge of $\mathcal{F}$.
Then for every pair $i,j\in V(\mathcal{F})$, $x_ix_j$ appears in at most one of  $\frac{\partial \lambda (\mathcal{F}, \vec{x})}{\partial x_i}$, $i\in [n]$.
Hence
$$ \sum_{i=1}^n \frac{\partial \lambda (\mathcal{F}, \vec{x})}{\partial x_i}\le \sum_{1\le i< j \le n} x_ix_j \le {n\choose 2}{1\over n^2}.$$
By Fact \ref{lemma-partion}, $ \frac{\partial \lambda (\mathcal{F}, \vec{x})}{\partial x_i}=3\lambda(\mathcal{F})$ for every $i\in [n]$.
Hence $3n\lambda(\mathcal{F})\le {n\choose 2}{1\over n^2}$, so $\lambda(\mathcal{F}) \le {n-1 \over 6n^2} \le {1 \over 32}$.
Note that $\mathcal{F}$ contains at least one edge, so by Fact \ref{mono}, $\lambda(\mathcal{F})\ge {1\over 27} > {1 \over 32}$, a contradiction.
\end{proof}

Denote $T_2=\{123,124\}$ and $F_5=\{123,124,345\}$. Note that $F_5=H^{T_2}$ and the following result in \cite{Sidorenko-89,BIJ} is a consequence of the above.
\begin{lemma}{\em (\cite{Sidorenko-89,BIJ})} \label{p_2}
Let $F\in \{P_2, T_2\}$ and $t=|V(F)|$.
 Then
$$\pi_{\lambda}(F)=3!\lambda(K_{t-1}^3).$$
Furthermore, for any $F$-free and $K_{t-1}^3$-free $3$-graph $G$, there exists $\epsilon=\epsilon(t)>0$ such that $\lambda(G)\le \lambda(K_{t-1}^3)-\epsilon$.
\end{lemma}
\begin{proof}
Let $F\in \{P_2, T_2\}$ and $t=|V(F)|$. Since $K_{t-1}^3$ is $F$-free, we have $\pi_{\lambda}(F)\ge 3!\lambda(K_{t-1}^3).$
Let $G$ be $F$-free and let $G^*$ be a subgraph of $G$ such that $G^*$ is dense and $\lambda(G^*)= \lambda(G)$.
By Lemma \ref{p_2'}, $|V(G)|\le t-1$ and hence $\lambda(G^*)\le \lambda(K_{t-1}^3)$.
Thus $\pi_{\lambda}(F)=3!\lambda(K_{t-1}^3).$
Furthermore, if $G$ is also $ K_{t-1}^3$-free, then $G^*$ is a copy of some subgraph of $K_{t-1}^{3-}$.
Note that $K_{3}^{3-}=\emptyset$.
Hence $\lambda(G^*)\le \lambda(K_{4}^3)-\epsilon$ for $t=5$, where $\epsilon=\lambda(K_{4}^3)-\lambda(K_{4}^{3-})={1\over 16}-{4\over 81}={17\over 1296}$, and $\lambda(G^*)=0$ for $t=4$.
\end{proof}

\subsection{Left-compressing $P_3$-free $3$-graphs covering pairs}

\begin{lemma} \label{lemmaP_3}
Let $\mathcal{F}$ be a $P_3$-free $3$-graph with $ n\ge 6$ vertices that covers pairs. Let $i,j\in V(\mathcal{F})$.
Then $\pi_{ij}(\mathcal{F})$ is also $P_3$-free. Furthermore, if $\mathcal{F}$ is also $K^3_{6}$-free, then $\pi_{ij}(\mathcal{F})$ is $K^3_{6}$-free.
\end{lemma}
\begin{proof}
Suppose for the contrary that there is a copy of $P_3$, denoted by $P$, such that $P \subseteq \pi_{ij}(\mathcal{F})$.
By the definition of $\pi_{ij}(\mathcal{F})$, for every edge $f\in \pi_{ij}(\mathcal{F})$ with $\{i,j\}\subseteq f$, $f \in \mathcal{F}$;
for every edge $f\in \pi_{ij}(\mathcal{F})$ with $j \in f$ and $i \notin f$, $f$ and $(f\setminus\{j\})\cup\{i\} \in \mathcal{F}$.
Since $\mathcal{F}$ is $P_3$-free, there is $e \in P$ such that $e \notin \mathcal{F}$ and $(e\setminus\{i\})\cup\{j\} \in \mathcal{F}$. There are two cases according to the degree of $i$ in $P$.

{\em Case 1.} $d_P(i)=1$. There are two subcases according to the degree of $j$ in $P$.

Subcase 1.1. $d_P(j)=0$. Then $(P\setminus e)\cup \{(e\setminus \{i\})\cup\{j\}\}$ is a copy of $P_3$ in $\mathcal{F}$, a contradiction.

Subcase 1.2. $d_P(j)= 1$ or $2$. Then $\{ f\cup \{i\}: f\in L_{P}(j \setminus i)\}\cup \{f\cup \{j\}: f\in L_{P}(i \setminus j)\}\cup \{e\in P: \text{ both} \: i,j \notin e\}$
(i.e., exchange $i$ and $j$ in $P$) is a copy of $P_3$ in $\mathcal{F}$, a contradiction.

{\em Case 2.} $d_P(i)=2$. There are three subcases according to the degree of $j$ in $P$.

Subcase 2.1. $d_P(j)=2$. We may assume $P=\{abi,icj,jdf\}$. Then $abj \in \mathcal{F}$ and $idf \in \mathcal{F}$. So $\{abj,jci,idf\}$ forms a copy of $P_3$ in $\mathcal{F}$, a contradiction.

Subcase 2.2. $d_P(j)=1$. If $ij$ is contained in some edge $f$ of $P$, then
 $\{(e\setminus \{i\})\cup\{j\},f,g\}$ forms a copy of $P$ in $\mathcal{F}$, where $g\in P\setminus\{e,f\}$, a contradiction.
Now assume that $P=\{abi,icd,dfj\}$. If $abj \in \mathcal{F}$, then we get a contradiction that $\{abj,dfj,icd\}$ forms a copy of $P_3$ in $\mathcal{F}$. Otherwise $jcd \in \mathcal{F}$. Note that $dfi \in \mathcal{F}$. Then we get a contradiction that $\{abi,ifd,dcj\}$ forms a copy of $P_3$ in $\mathcal{F}$.

Subcase 2.3. $d_P(j)=0$. We can assume that $P=\{abi,icd,dgh\}$. If both $abj,jcd$ are in $\mathcal{F}$, then $\{abj,jcd,dgh\}$ forms a copy of $P_3$ in $\mathcal{F}$, a contradiction.
Otherwise we get a copy of $\{a_1a_2b_0,b_0c_2c_1,b_1b_2b_3\}$ in $\mathcal{F}$. Relabel the vertices of $\mathcal{F}$ such that $\{a_1a_2b_0,b_0c_2c_1,b_1b_2b_3\} \subseteq \mathcal{F}$.
Since $\mathcal{F}$ covers pairs, then for every $p\in[2]$ and $q\in [3]$, there exists $x_{pq} \in V(\mathcal{F})$  such that $a_pb_qx_{pq} \in \mathcal{F}$.
Then $x_{pq}=a_k$ for $k\in \{1,2\}\setminus \{p\}$, i.e., $a_1a_2b_1,a_1a_2b_2,a_1a_2b_3 \in \mathcal{F}$, since otherwise there exists $p\in[2]$ and $q\in [3]$ such that
 $\{b_0c_2c_1,a_pb_qx_{pq},b_1b_2b_3\}$ ($x_{pq}\in \{c_1,c_2\}$) or $\{a_1a_2b_0,b_0c_2c_1,a_pb_qx_{pq}\}$ ($x_{pq}\notin \{c_1,c_2\}$)
forms a copy of $P_3$ in $\mathcal{F}$.
In a similar way, $c_1c_2b_1,c_1c_2b_2,c_1c_2b_3 \in \mathcal{F}$.
Then $\{a_1a_2b_1,b_1b_2b_3,b_3c_1c_2\}$ forms a copy of $P_3$ in $\mathcal{F}$, a contradiction.

Next, suppose that $\mathcal{F}$ is  $K^3_{6}$-free. Since $\mathcal{F}$ covers pairs, $\{i,j\}$ is contained in some edge $g$ of $\mathcal{F}$.
Suppose for contradiction that $\pi_{ij}(\mathcal{F})$ contains a copy $K$ of $K^3_{6}$. Clearly $V(K)$ must contain $i$. If $V(K)$ also contains $j$ then it is easy to see that $K$ also exists in $\mathcal{F}$, contradicting $\mathcal{F}$ being $K^3_{6}$-free.
All the edges in $K$ not containing $i$ also exist in $\mathcal{F}$.
Without loss of generality assume that $[4]\subseteq V(K)\setminus g$.
Thus $\{123,34i,g\}$ or $\{123,34j,g\}$ forms a copy of $P_3$ in $\mathcal{F}$, a contradiction.
\end{proof}

\subsection{Left-compressing  $P_4$-free $3$-graphs covering pairs}

Denote $F_1$ as the $3$-graph with the union of two disjoint $P_2$'s, $F_2$ as the $3$-graph with the union of disjoint $P_1$ and $P_3$,
and $F_3$ as the $3$-graph with vertex set $\{a_1,a_2,a_3,b_1,b_2,c_1,c_2,d_1,d_2\}$ and edge set $\{a_1a_2a_3,a_1b_1b_2,a_2c_1c_2$, $a_3d_1d_2\}$.
The following lemma is vital for the proof of the Lagrangian density of $P_4$ and its proof is postponed to Subsection 3.4.
\vspace{1em}
\begin{figure}[H]
\centering
\begin{minipage}{5cm}
\includegraphics[width=1\textwidth, height=0.35\textwidth]{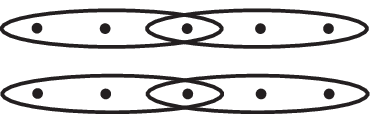}
\label{fig F1}
\caption{$F_1$}
\end{minipage}
\begin{minipage}{5cm}
\includegraphics[width=1\textwidth, height=0.35\textwidth]{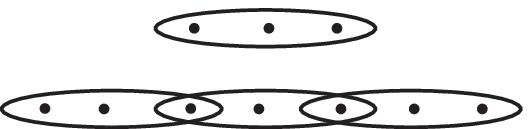}
\label{fig F2}
\caption{$F_2$}
\end{minipage}
\begin{minipage}{5cm}
\includegraphics[width=0.9\textwidth, height=0.7\textwidth]{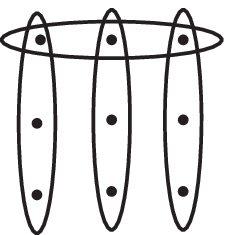}
\label{fig F3}
\caption{$F_3$}
\end{minipage}
\end{figure}

\begin{lemma} \label{lemmaF_i}
If $\mathcal{F}$ be a $P_4$-free $3$-graph with $ n\ge 9$ vertices that covers pairs, then $\mathcal{F}$ is $F_1$-free and $F_2$-free. Furthermore, if $\lambda(\mathcal{F}) \ge \lambda(K_{8}^3)-0.005$, then $\mathcal{F}$ is also $F_3$-free.
\end{lemma}

\begin{lemma} \label{lemmaP_4}
Let $\mathcal{F}$ be a $P_4$-free $3$-graph with $ n\ge 9$ vertices that covers pairs.
If $\lambda(\mathcal{F}) \ge \lambda(K_{8}^3)-0.005$, then $\pi_{ij}(\mathcal{F})$ is also $P_4$-free.
Furthermore, if $\mathcal{F}$ is also $K^3_{8}$-free, then $\pi_{ij}(\mathcal{F})$ is $K^3_{8}$-free.
\end{lemma}
\begin{proof}
The proof is similar to Lemma \ref{lemmaP_3}.
Suppose for the contrary that there is a copy of $P_4$, denoted by $P'$, such that $P' \subseteq \pi_{ij}(\mathcal{F})$.
Since $\mathcal{F}$ is $P_4$-free then there is $e^* \in P'$ such that $i \in e^*$, $e^* \notin \mathcal{F}$ and $(e^*\setminus\{i\})\cup\{j\} \in \mathcal{F}$.

{\em Case 1.} $d_{P'}(i)=1$. If $d_{P'}(j)=0$. Then $(P'\setminus e^*)\cup \{(e^*\setminus \{i\})\cup\{j\}\}$ forms a copy of $P_4$ in $\mathcal{F}$, a contradiction.
Otherwise $d_{P'}(j)= 1$ or $2$. Then $\{ f^*\cup \{i\}: f^*\in L_{P'}(j \setminus i)\})\cup \{f^*\cup \{j\}: f^*\in L_{P'}(i \setminus j)\})\cup \{f^*\in P':\{i,j\}\subseteq f^* \: {\text or \: both} \: i,j \notin f^*\}$ contains a copy of $P_4$ in $\mathcal{F}$, a contradiction.

{\em Case 2.} $d_{P'}(i)=2$.

Subcase 1. $d_{P'}(j)=0$.
There are two subcases: $P'=\{abi,icd,def,fgh\}$ or $P'=\{abc,cdi,ief,fgh\}$.
First we consider $P'=\{abi,icd,def,fgh\}$. If both $abj,jcd \in \mathcal{F}$, then $\{abj,jcd,def,fgh\}$ forms a copy of $P_4$ in $\mathcal{F}$, a contradiction.
Otherwise we get $\{abi,jcd,def,fgh\}$ or $\{abj,icd,def,fgh\}$ in $\mathcal{F}$, which is isomorphic to $F_2$. By Lemma \ref{lemmaF_i}, this is a contradiction.
Suppose $P'=\{abc,cdi,ief,fgh\}$.
If both $cdj, jef \in \mathcal{F}$, then $\{abc,cdj,jef,fgh\}$ forms a copy of $P_4$ in $\mathcal{F}$, a contradiction. Otherwise we get $\{abc,cdi,jef,fgh\}$ or $\{abc,cdj,ief,fgh\}$ in $\mathcal{F}$, which is isomorphic to $F_1$. By Lemma \ref{lemmaF_i}, this is a contradiction.

Subcase 2. $d_{P'}(j)=1$. If $ij$ is contained in some edge $g^*$ of $P'$, then $(h^*\setminus \{i\})\cup\{j\} \in \mathcal{F}$ for some $h^* \in P'$ with $i\in h^*$ and $j \notin h^*$. Hence $\{(h^*\setminus \{i\})\cup\{j\},g^*\})\cup (\{P\setminus\{g^*,h^*\})$ forms a copy of $P_4$ in $\mathcal{F}$, a contradiction.
Otherwise then $P'=\{abi,icd,djf,fgh\}$ or $\{abi,icd,dfg,gjh\}$ or $\{jab,bci,idf,fgh\}$.
Denote the edges $e_1,e_2,e_3 \in P'$ such that $i\in e_1,e_2$ and $j\in e_3$.
If both $(e_1\setminus \{i\})\cup\{j\}, (e_2\setminus \{i\})\cup\{j\} \in \mathcal{F}$, then $\{ (e_1\setminus \{i\})\cup\{j\}, (e_2\setminus \{i\})\cup\{j\}, (e_3\setminus \{j\})\cup\{i\} \} \cup (P\setminus\{e_1,e_2,e_3\})$ forms a copy of $P_4$ in $\mathcal{F}$, a contradiction.
So assume exactly one of $(e_1\setminus \{i\})\cup\{j\}, (e_2\setminus \{i\})\cup\{j\}$ is not in $\mathcal{F}$.

For $P'=\{abi,icd,dfj,fgh\}$, if $abj \in \mathcal{F}$ then $\{abj,icd,dfj,fgh\}$ forms a copy of $F_3$ in $\mathcal{F}$, by Lemma \ref{lemmaF_i}, this is a contradiction.
If $jcd \in \mathcal{F}$ then $\{dfi,jcd,fgh,abi\}$ forms a copy of $F_3$ in $\mathcal{F}$, a contradiction.

For $P'=\{abi,icd,dfg,jgh\}$, if $abj \in \mathcal{F}$ then $\{icd,dfg,jgh,abj\}$ forms a copy of $P_4$ in $\mathcal{F}$, a contradiction.
If $jcd \in \mathcal{F}$ then $\{jcd,dfg,ghi,iab\}$ forms a copy of $P_4$ in $\mathcal{F}$, a contradiction.

For $P'=\{abj,ibc,idf,fgh\}$, if $jbc \in \mathcal{F}$ then $\{jbc,abi,idf,fgh\}$ forms a copy of $P_4$ in $\mathcal{F}$, a contradiction.
If $jdf \in \mathcal{F}$ then $\{abj,bci,jdf,fgh\}$ forms a copy of $P_4$ in $\mathcal{F}$, a contradiction.

Subcase 3. $d_{P'}(j)=2$.
If $\{i,j\}$ is contained in some edge $e_4$ of $P'$, denote the edge of $P'$ containing $i$ but not $j$ as $e_5$,
the edge of $P'$ containing $j$ but not $i$ as $e_6$, and the edge of $P'$ containing neither $i$ nor $j$ as $e_7$,
then $\{(e_6\setminus \{j\})\cup \{i\}, e_7,e_4, (e_5\setminus \{i\})\cup \{j\}\}$ forms a copy of $P_4$ in $\mathcal{F}$, a contradiction.
Otherwise we can assume that $P'=\{abi,icd,dej,jfg\}$.
If both $abj,jcd$ in $\mathcal{F}$, then $\{abj,jcd,dei,ifg\}$ forms a copy of $P_4$ in $\mathcal{F}$, a contradiction.
If $abj \notin \mathcal{F}$ and $jcd \in \mathcal{F}$, then $\{abi,dei,jcd,jgh\}$ forms a copy of $P_4$ in $\mathcal{F}$, a contradiction.
Otherwise $abj \in \mathcal{F}$ and $jcd \notin \mathcal{F}$, then $\{abj,dej,icd,ifg\}$ forms a copy of $P_4$ in $\mathcal{F}$, a contradiction.

Next, suppose that $\mathcal{F}$ is  $K^3_{8}$-free. Since $\mathcal{F}$ covers pairs, $\{i,j\}$ is contained in some edge $g^*$ of $\mathcal{F}$.
Suppose for contradiction that $\pi_{ij}(\mathcal{F})$ contains a copy $K$ of $K^3_{8}$. Clearly, $V(K)$ must contain $i$. If $V(K)$ also contains $j$ then it is easy to see that $K$ also exists in $\mathcal{F}$, contradicting $\mathcal{F}$ being $K^3_{8}$-free.
By our assumption, $V(K)$ contains at least $6$ vertices outside $g^*$. Without loss of generality assume that $[6]\subseteq V(K)\setminus g^*$.
Thus $\{123,345,56i,g^*\}$ or $\{123,345,56j,g^*\}$ forms a copy of $P_4$ in $\mathcal{F}$, a contradiction.
\end{proof}

\subsection{Lagrangian densities of $P_3$ and $P_4$}
We will perform the following  algorithm in the proof of Theorem \ref{theoremP_3P_4}.
In the  algorithm, $t=3$ or $4$.
\begin{algorithm}\label{left-compression1}{\rm (Dense and left-compressed $3$-graph)}

\noindent{\bf Input:}   A $P_t$-free $3$-graph $G$ on vertex set $[n]$ ( with $\lambda(G)\ge \lambda(K_{8}^3)-0.005$  for $t=4$).

\noindent{\bf Output:}  A $P_t$-free $3$-graph $G'$ that is dense and left-compressed, and satisfies that $\lambda(G')\ge\lambda(G)$.

\noindent{\bf Step 1.} If $G$ is dense then let $G'=G$ and go to Step 2.
 Otherwise replace $G$ by a dense subgraph $G'$ with the same Lagrangian and go to Step 2.

\noindent{\bf Step 2.} Let $\vec{x}=(x_1,\dots,x_n)$ be an optimum weighting of $G'$.
Assume that $x_1\ge x_2\ge ... \ge x_{|V(G')|}$ since otherwise we can relabel the vertices.
If $G'$ is left-compressed then terminate. Otherwise there exist vertices  $i,j$, where $i<j$, such that $L_{G'}(j \setminus i)\neq \emptyset $, then replace $G'$ by $\pi_{ij}(G')$ and go to step 1.
\end{algorithm}

Note that the algorithm terminates after finite many steps. By Lemma \ref{lemmaP_3} or \ref{lemmaP_4} and Fact \ref{compression-preserve}. The output of the algorithm  $G'$ is a dense and left-compressed $P_t$-free $3$-graph with Lagrangian at lest  $\lambda(G)$.

\begin{theo}\label{theoremP_3P_4}
Let $t=3$ or $4$.
Then
$$\pi_{\lambda}(P_t)=3!\lambda(K_{2t}^3).$$
Furthermore, for any $P_t$-free and $ K_{2t}^3$-free $3$-graph $G$, there exists $\epsilon=\epsilon(t)>0$ such that $\lambda(G)\le \lambda(K_{2t}^3)-\epsilon$.
\end{theo}
\begin{proof}
Let $t=3$ or $4$.
Let $F$ be a  $P_t$-free $3$-graph with with $\lambda(F)\ge \lambda(K_{8}^3)-0.005$.
Applying  Algorithm \ref{left-compression1}  to $F$ and let $G$ be the final hypergraph obtained. 
Thus $G$ is a $P_t$-free $3$-graph that is dense and left-compressed, and $\lambda(G)\ge \lambda(K_{8}^3)-0.005$.
Assume that $V(G)=[n]$.
If $n\le 2t$ then $\lambda(G) \le \lambda(K_{2t}^3)$.
Furthermore, if $F$ is $K_{2t}^3$-free then $G$ is $K_{2t}^3$-free by Lemma \ref{lemmaP_3} and  Lemma \ref{lemmaP_4},
and hence $\lambda(G)\le \lambda(K_{2t}^{3-})<\lambda(K_{2t}^3)$.
Now suppose that $n \ge 2t+1$.
We use induction on $t\ge 2$. The base case $t=2$ is guaranteed by Lemma \ref{p_2}. Suppose that the result holds for $t-1$.
 Let $\vec{x}=(x_1,\dots,x_n)$ be an optimum weighting of $G$. Denote $L(1)=\{e \in G: 1 \in e\}$. Then
$$ \lambda(G)=\lambda(G,\vec{x})=\lambda(L(1),\vec{x})+\lambda(G[[2,n]],\vec{x'}), $$
where $\vec{x'}=(x_2,\dots,x_n)$.
Clearly, $\lambda(L(1),\vec{x})\le x_1\sum_{2\le i<j\le n}x_ix_j \le {1 \over 2}x_1(1-x_1)^2$.
For the second term, we divide it into two cases according to $t=3$ or $4$.

Case $t=3$.
We claim that $G[[2,n]]=\emptyset$. Otherwise since $G$ is left-compressed and so does $G[[2,n]]$ (on vertex set $[2,n]$), then
$234\in G[[2,n]]$. $G$ covers pairs and $n\ge 7$ implies that $167\in G[[2,n]]$ and $125 \in G[[2,n]]$.
Then $\{167,125,234\}$ forms a copy of $P_3$, a contradiction.
Hence
$$
 \lambda(G)= \lambda(L(1),\vec{x})+\lambda(G[[2,n]],\vec{x})
 \le {1 \over 2}x_1(1-x_1)^2
 \le {1\over 4}\left({2x_1+1-x_1+1-x_1\over 3}\right)
 = {1 \over 4}\left({2\over 3}\right)^3
 ={2\over 27}.
$$
Note that by Fact \ref{K_t^r} and \ref{K_6^3-}, we have $\lambda(K_6^3)={5\over 54}$ and $\lambda(K_6^{3-})=0.0887$.
Let $$\epsilon(3)=\min\left\{{5\over 54}-0.0877, {5\over 54}-{2\over 27}\right\}>0.0048$$ and we are done for the case $t=3$.

Case $t=4$. We claim that $G[[2,n]]$ is $P_{3}$-free. Otherwise suppose that there is a copy of $P_{3}$, denoted by $P'$, in $G[[2,n]]$.
Since $n\ge 9$ and $|V(P')|=7$, there is $v \in \{2,\dots,n\} \setminus V(P')$.
Let $u\in V(P')$ such that $d_{P'}(u)=1$.
Since $G$ covers pairs and $G$ is left-compressed, we have $1(n-1)n\in G$. Then $1uv \in G$.
Hence $P' \cup \{1uv\}$ forms a copy of $P_4$ in $G$, a contradiction.

Let $H$ be a dense subgraph of $G[[2,n]]$ with $\lambda(H)=\lambda(G[[2,n]])$. Note that $H$ must be an induced subgraph of $G[[2,n]]$.
Every weight in an optimum weighting $\vec{y}$ of $H$ is positive. If $H$ is not an induced subgraph of $G[[2,n]]$, let $H'$ be an induced subgraph of $G[[2,n]]$ on $V(H)$, then $\lambda(G[[2,n]])\ge \lambda(H',\vec{y})>\lambda(H,\vec{y})=\lambda(G[[2,n]])$, a contradiction.
Note that $\lambda(G[[2,n]])$ is left-compressed on $[2,n]$.
Suppose $V(H)=\{i_1,\dots,i_{s-1}\}$, where $i_1<\dots<i_{s-1}$ and $s$ is a positive integer.
For each $i_{j_1}i_{j_2}i_{j_3} \in H$ with $j_1<j_2<j_3$, since $G[[2,n]]$ is left-compressed on $[2,n]$ and $H$ is an induced subgraph of $G[[2,n]]$,  $i_{k_1}i_{k_2}i_{k_3} \in H$ for every $\{i_{k_1},i_{k_2},i_{k_3}\} \subseteq V(H)$ satisfying $j_l \ge k_l$ for every $l\in [3]$.
Relabel the vertex $i_j$ of $H$ with $j+1$ for each $j\in [s-1]$, then $H$ is left-compressed on vertex set $[2,s]$.

If $s\ge 8$, then $\lambda(H)\le {2\over 27}$ following from the case of $t=3$ with the number of vertices not less than $7$.
Otherwise $s\le 7$.
 When $s\le 6$, then $\lambda(H)\le \lambda(K_5^3)={2\over 25}$.
 Now assume that $s= 7$, we claim that $346 \notin H$. Otherwise suppose that $346 \in H$. Since $H$ is left-compressed, $\{346,235,127,189\}$ forms a copy of $P_4$ in $G$, a contradiction.
 Thus all $ijk$ with $i\ge 3$, $j\ge 4$, $k\ge 6$ and $i+j+k>13$ are not in $H$.
 Then $H \subseteq \{234,235,245,345,236,246,256,345,346,356,456,237,247,257,267\}:=H^*$.
 Since the maximum clique of $H^*$ is on $[2,6]$ and the number of edges of $H^*$ is $15$,
 then by Theorem \ref{theoPZ}, we have $\lambda(H^*)= \lambda(K_5^3)={2\over 25}$.
 Hence $\lambda(G[[2,n]])\le \max\left\{{2\over 27},{2\over 25}\right\}={2\over 25}.$
 Then
\begin{eqnarray*}
 \lambda(G)&=& \lambda(L(1),\vec{x})+\lambda(G[[2,n]],\vec{x})\\
 &\le& {1 \over 2}x_1(1-x_1)^2+{2\over 25}(1-x_1)^3\\
 &=&{1\over 100}{1\over 21^2}(21-21x_1)^2(42x_1+8)\\
 &\le&{1\over 100}{1\over 21^2}\left({21+21+8\over 3}\right)^3\\
 &=&{1250\over 11907}.
\end{eqnarray*}
Note that by Fact \ref{K_t^r} and \ref{K_8^3-}, we have $\lambda(K_8^3)={7\over 64}$ and $\lambda(K_8^{3-})=0.1077$.
 Let
$$\epsilon(4)=\min\left\{{7\over 64}-0.1077, {7\over 64}-{1250\over 11907}\right\}\ge 0.0016$$ and we are done for the case $t=4$.
\end{proof}

\subsection{Proof of Lemma \ref{lemmaF_i}}

Given a $3$-graph $\mathcal{F}$ that covers pairs and $a, b\in V(\mathcal{F})$, let {\em Cover$(a, b)$} denote the property that $\{a, b\}$ is covered by $\mathcal{F}$
and $N_{ab}=\{v\in V(\mathcal{F}): abv \in \mathcal{F}\}$.
We repeatly use the property that $\mathcal{F}$ covers pairs to prove Lemma \ref{lemmaF_i}.
\medskip
\newline
{\em Lemma \ref{lemmaF_i}.
Let $\mathcal{F}$ be a $P_4$-free $3$-graph with $ n\ge 9$ vertices that covers pairs. Then $\mathcal{F}$ is $F_1$-free and $F_2$-free. Furthermore, if $\lambda(\mathcal{F}) \ge \lambda(K_{8}^3)-0.005$, then $\mathcal{F}$ is also $F_3$-free.
}
\medskip

\begin{proof}
Note that since $\mathcal{F}$ covers pairs, we have $N_{uv}\neq \emptyset$ for every pair of vertices $u,v\in V(\mathcal{F})$.

(i) First suppose that $\mathcal{F}$ contains a copy of $F_1$, denoted as $F$. Denote  $V(F)=\{a, b_1, b_2, b_3, b_4, $ $c_1, c_2, c_3, c_4, d\}$ and $E(F)=\{c_1c_2a, ac_3c_4, b_1b_2d, db_3b_4\}$.

Claim 1. $N_{b_ic_j}\subseteq \{b_{i'}, c_{j'}\}$, where $i, j, i',j'$ satisfy $\{i, i'\}=\{1, 2\}$ or $\{3, 4\}$ and $\{j, j'\}=\{1, 2\}$ or $\{3, 4\}$.
By symmetry, we only prove that $N_{b_1c_1} \subseteq \{b_2, c_2\}$.
If $a\in N_{b_1c_1}$, then $\{c_3c_4a,ac_1b_1,b_1b_2d,$ $db_3b_4\}$ forms a copy of $P_4$, a contradiction. Similarly, all $d,c_3,c_4,b_3,b_4$ and vertices in $V(\mathcal{F})\setminus V(F)$ are not in $N_{b_1c_1}$.
Thus we complete the proof of the claim.

By Claim 1, $N_{b_1c_1}\neq \emptyset$ and $N_{b_1c_1} \subseteq \{b_2, c_2\}$, then $b_1b_2c_1 \in \mathcal{F}$ or $b_1c_1c_2 \in \mathcal{F}$.
By symmetry, assume that $b_1b_2c_1 \in \mathcal{F}$.
Then $b_3b_4c_3 \notin \mathcal{F}$.
 Otherwise $\{b_1b_2c_1, c_1c_2a,ac_3c_4,b_3b_4c_3 \}$ forms a copy of $P_4$, a contradiction.
So by Claim 1, Cover$(b_3,c_3)$ implies that $b_3c_3c_4 \in \mathcal{F}$.
Now we show that $b_3b_4c_2 \notin \mathcal{F}$. Otherwise $\{b_1b_2c_1, c_1c_2a, b_3b_4c_2, b_3c_3c_4 \}$ forms a copy of $P_4$, a contradiction.
Thus by Claim 1, Cover$(b_4,c_2)$ implies that $ b_4c_1c_2 \in \mathcal{F}$.
Then $\{b_3c_3c_4, db_3b_4, b_4c_1c_2, b_1b_2c_1\}$ forms a copy of $P_4$, a contradiction.

\medskip

(ii) Suppose that $\mathcal{F}$ contains a copy of $F_2$, denoted as $F'$.
Denote  $V(F')=\{a_1, a_2, a_3, c_1, c_2, c_3, c_4, $ $b_1, b_0, b_2\}$ and $E(F')=\{a_1a_2a_3, c_1c_2b_1, b_1b_0b_2, b_2c_3c_4\}$.

{\em Case 1.} $c_1c_2a_i, c_3c_4a_i \in \mathcal{F}$ for all $i\in [3]$.
Cover$(c_2,c_3)$ implies that either $c_1c_2c_3 \in \mathcal{F}$ or $c_2c_3c_4 \in \mathcal{F}$.
Indeed, $a_i \notin N_{c_2c_3}$ for each $i\in [3]$, otherwise without loss of generality suppose that $c_2c_3a_1 \in \mathcal{F}$, then $\{a_1a_2a_3,a_1c_2c_3,c_3c_4b_2,b_2b_0b_1\}$ forms a copy of $P_4$ in $\mathcal{F}$, a contradiction.
$b_1, b_2\notin N_{c_2c_3}$, otherwise without loss of generality assume that $c_2c_3b_1 \in \mathcal{F}$, then
$\{b_0b_1b_2, b_1c_2c_3, c_3c_4a_3, a_3a_2a_1\}$ forms a copy of $P_4$ in $\mathcal{F}$, a contradiction.
$b_0 \notin N_{c_2c_3}$, otherwise
$\{a_3a_2a_1,$ $c_1c_2a_1, c_2c_3b_0, b_0b_1b_2\}$ forms a copy of $P_4$ in $\mathcal{F}$, a contradiction.
$ x\notin N_{c_2c_3}$ for some vertex $x\in V(\mathcal{F})\setminus V(F)$, otherwise
$\{b_0b_1b_2, b_2c_3c_4, c_2c_3x, c_1c_2a_1\}$ forms a copy of $P_4$ in $\mathcal{F}$, a contradiction.
Therefore either $c_1c_2c_3 \in \mathcal{F}$ or $c_2c_3c_4 \in \mathcal{F}$.
Without loss of generality assume that $c_1c_2c_3 \in \mathcal{F}$.

Consider the property Cover$(a_1, b_0)$.
$a_2, a_3 \notin N_{a_1b_0} $, otherwise by symmetry, assume that $a_1b_0a_2 \in \mathcal{F}$, then
$\{a_1b_0a_2, b_0b_1b_2, b_2c_3c_4, c_1c_2c_3\}$ forms a copy of $P_4$ in $\mathcal{F}$, a contradiction.
$ b_1, b_2 \notin N_{a_1b_0}$, otherwise by symmetry, assume that $a_1b_0b_1 \in \mathcal{F}$, then
$\{c_3c_4a_3, a_3a_2a_1, a_1b_0b_1, b_1c_1c_2\}$ forms a copy of $P_4$ in $\mathcal{F}$, a contradiction.
$c_i \notin N_{a_1b_0}$ for each $i\in [4]$, otherwise by symmetry, assume that $a_1b_0c_1 \in \mathcal{F}$, then
$\{c_3c_4a_3, a_3a_2a_1, $ $a_1b_0c_1, b_1c_1c_2\}$ forms a copy of $P_4$ in $\mathcal{F}$, a contradiction.
$x \notin N_{a_1b_0}$ for some vertex $x\in V(\mathcal{F})\setminus V(F)$, otherwise
$\{a_3a_2a_1, a_1b_0x, b_0b_1b_2, b_2c_3c_4\}$ forms a copy of $P_4$ in $\mathcal{F}$, a contradiction.
Hence $\{a_1, b_0\}$ is not covered by any edge of $\mathcal{F}$, which contradicts that $\mathcal{F}$ covers pairs.

{\em Case 2.}
$c_1c_2a_i$ or $c_3c_4a_i \notin \mathcal{F}$ for some $i \in [3]$.
Without loss of generality assume that $c_1c_2a_1 \notin \mathcal{F}$.
Consider the property Cover$(a_1, c_1)$. Then $N_{a_1c_1}=\{b_2\}$.
Since if $a_2\in N_{a_1c_1}$, then $\{a_1a_2c_1, c_1c_2b_1, b_1b_0b_2,$ $b_2c_3c_4\}$ forms a copy of $P_4$ in $\mathcal{F}$, a contradiction.
Similarly, $a_3\notin N_{a_1c_1}$.
If $b_1\in N_{a_1c_1}$, then $\{a_3a_2a_1,$ $a_1c_1b_1, b_1b_0b_2,b_2c_3c_4\}$ forms a copy of $P_4$ in $\mathcal{F}$, a contradiction.
Similarly, $b_0,c_3,c_4\notin N_{a_1c_1}$.
If $x\in N_{a_1c_1}$ for some vertex $x\in V(\mathcal{F})\setminus V(F')$, then $\{a_3a_2a_1,a_1xc_1,c_1c_2b_1,b_1b_0b_2\}$ forms a copy of $P_4$ in $\mathcal{F}$, a contradiction.
Similarly, $N_{a_1c_2}=\{b_2\}$.
We claim that $c_3c_4a_2, c_3c_4a_3 \notin \mathcal{F}$; otherwise without loss of generality assume that $c_3c_4a_2 \in \mathcal{F}$, then
$\{c_3c_4a_2, a_3a_2a_1, a_1c_1b_2, b_1b_2b_3\}$ forms a copy of $P_4$ in $\mathcal{F}$, a contradiction.

Hence $c_1c_2a_1 \notin \mathcal{F}$ implies that $a_1c_1b_2,a_1c_2b_2 \in \mathcal{F}$ and $c_3c_4a_2,c_3c_4a_3 \notin \mathcal{F}$.
Similarly, $c_3c_4a_2,c_3c_4a_3 \notin \mathcal{F}$ imply $c_1c_2a_3,c_1c_2a_2  \notin \mathcal{F}$ and $c_1c_2a_2  \notin \mathcal{F}$ implies that $c_3c_4a_1  \notin \mathcal{F}$. Thus
$c_1c_2a_i,c_3c_4a_i \notin \mathcal{F}$ for each $i\in[3]$.
Then $N_{a_ic_j}=\{b_{k_j}\}$ for every $i\in [3],j\in [4]$, where $k_1,k_2=2$ and $k_3,k_4=1$.
Thus
$\{a_3b_1c_3, a_2b_1c_4, a_2b_2c_1, a_1b_2c_2\}$ forms a copy of $P_4$ in $\mathcal{F}$, a contradiction.
So $ \mathcal{F}$ is $F_2$-free.
\medskip

(iii) Let $\mathcal{F}$ satisfy $\lambda(\mathcal{F}) \ge \lambda(K_{8}^3)-0.005$.
Suppose that $\mathcal{F}$ contains a copy of $F_3$, denoted as $F''$.
Let  $V(F'')=\{a_1, a_2, a_3, c_1, c_2, b_1, b_2, d_1, d_2\}$ and $E(F'')=\{a_1a_2a_3, a_1b_1b_2, a_2c_1c_2,  a_3d_1d_2\}$.
We first prove that $V(\mathcal{F})=V(F'')$. Otherwise suppose that $V(\mathcal{F})\setminus V(F_3)\neq \emptyset$, let
$x \in V(\mathcal{F})\setminus V(F_3)$.
The following claim is a simple consequence of $\mathcal{F}$ being $P_4$-free.


We claim that $xb_1b_2, xc_1c_2, xd_1d_2 \notin \mathcal{F}$.
To prove this, we first assume that there are at least two of $xb_1b_2, xc_1c_2, xd_1d_2$ in $\mathcal{F}$.
Suppose $xb_1b_2, xc_1c_2\in \mathcal{F}$.
If $b_1\in N_{xd_2}$, then $\{xb_1d_2, d_2d_1a_3,a_3a_1a_2,a_2c_1c_2\}$ forms a copy of $P_4$ in $\mathcal{F}$, a contradiction.
Similarly, all $b_2, c_1, c_2$ and the vertices in $V(\mathcal{F})\setminus (V(F_3)\cup \{x\})$ (if there exists) are not in $N_{xd_2}$.
If $a_3\in N_{xd_2}$, then $\{b_1b_2x, xd_2a_3, a_1a_2a_3, a_2c_1c_2\}$ forms a copy of $P_4$ in $\mathcal{F}$, a contradiction.
If $a_1\in N_{xd_2}$, then $\{b_1b_2x, xd_2a_1, a_1a_2a_3, a_2c_1c_2\}$ forms a copy of $P_4$ in $\mathcal{F}$, a contradiction.
Similarly, $a_2\notin N_{xd_2}$.
Hence $N_{xd_2} = \{d_1\}$.

Now we show that $\{b_2, c_2\}$ is not covered by any edge of $\mathcal{F}$ and we get a contradiction.
$b_1 \notin N_{b_2c_2}$since otherwise $\{d_2d_1a_3,a_3a_1a_2,a_2c_1c_2, b_1b_2c_2\}$ forms a copy of $P_4$ in $\mathcal{F}$, a contradiction.
Similarly, all $c_1$ and the vertices in $V(\mathcal{F})\setminus (V(F_3)\cup \{x\})$ (if there exists) are not in $N_{b_2c_2}$.
$a_1 \notin N_{b_2c_2}$ since otherwise $\{b_1b_2x, b_2c_2a_1, a_1a_2a_3, a_3d_1d_2\}$ forms a copy of $P_4$ in $\mathcal{F}$, a contradiction.
Similarly, $a_2 \notin N_{b_2c_2}$.
$a_3 \notin N_{b_2c_2}$ since otherwise $\{d_1d_2x, b_1b_2x, b_2c_2a_3, a_1a_2a_3\}$ forms a copy of $P_4$ in $\mathcal{F}$, a contradiction.
$d_1 \notin N_{b_2c_2}$ since otherwise $\{b_1b_2x, b_2c_2d_1, a_3d_1d_2, a_1a_2a_3\}$ forms a copy of $P_4$ in $\mathcal{F}$, a contradiction.
Similarly, $d_2 \notin N_{b_2c_2}$.
Thus, $\{b_2, c_2\}$ is not covered by any edge of $\mathcal{F}$, a contradiction.

Now assume that there is exactly one of $xb_1b_2, xc_1c_2, xd_1d_2$ in $\mathcal{F}$. Without loss of generality assume that $xb_1b_2 \in \mathcal{F}$ and $xc_1c_2, xd_1d_2 \notin \mathcal{F}$.
Consider the property Cover$\{x, c_2\}$.
$b_1 \notin N_{xc_2}$since otherwise $\{d_2d_1a_3,a_3a_1a_2,a_2c_1c_2, c_2b_1x\}$ forms a copy of $P_4$ in $\mathcal{F}$, a contradiction.
Similarly, all $b_2, d_1, d_2, c_1$ and the vertices in $V(\mathcal{F})\setminus (V(F_3)\cup \{x\})$ (if there exists) are not in $N_{xc_2}$.
$a_1 \notin N_{xc_2}$since otherwise $\{b_1b_2x, xc_2a_i, a_1a_2a_3, a_3d_1d_2\}$ forms a copy of $P_4$ in $\mathcal{F}$, a contradiction.
Hence $N_{xc_2}\neq \emptyset$ implies that $y_{xc_2} = a_3$, that is, $xc_2a_3 \in \mathcal{F}$.
Similarly, $a_2 \notin N_{xc_2}$.
Then $N_{xc_2}= \{a_3\}$.
Similarly, $N_{xd_2}= \{a_2\}$.

Consider the property Cover$\{b_2,c_2\}$.
$b_1 \notin N_{b_2c_2}$since otherwise $\{d_2d_1a_3,a_3a_1a_2,a_2c_1c_2, c_2b_1b_2\}$ forms a copy of $P_4$ in $\mathcal{F}$, a contradiction.
Similarly, all $c_1, x$ and the vertices in $V(\mathcal{F})\setminus (V(F_3)\cup \{x\})$ (if there exists) are not in $N_{b_2c_2}$.
$a_1 \notin N_{xc_2}$since otherwise $\{b_1b_2x, b_2c_2a_1, a_1a_2a_3, a_3d_1d_2\}$ forms a copy of $P_4$ in $\mathcal{F}$, a contradiction.
Similarly, $a_2 \notin N_{b_2c_2}$.
$d_1 \notin N_{xc_2}$since otherwise $\{xb_1b_2, b_2c_2d_1, d_1d_2a_3, a_1a_2a_3\}$ forms a copy of $P_4$ in $\mathcal{F}$, a contradiction.
Similarly, $d_2 \notin N_{b_2c_2}$.
Hence $N_{b_2c_2}=\{a_3\}$. Then
$\{a_1b_1b_2, b_2c_2a_3, a_3d_1d_2, d_2a_2x\}$ forms a copy of $P_4$ in $\mathcal{F}$, a contradiction.
Hence $xb_1b_2, xc_1c_2, xd_1d_2 \notin \mathcal{F}$.

Now we claim that for every  $A \in \{ xb_ia_2,xb_ia_3, xc_ia_1,xc_ia_3, xd_ia_1,xd_ia_2:i\in [2]\}$, we have $A \notin \mathcal{F}$. By symmetry, suppose that $xb_2a_2 \in \mathcal{F}$.
Consider the property Cover$\{b_1, c_1\}$. We first show that all $d_1,d_2,a_3$ are not in $N_{b_1c_1}$; otherwise, denote such an edge by $e$, then
$\{xb_2a_2, a_2c_1c_2, e, d_1d_2a_3\}$ forms a copy of $P_4$ in $\mathcal{F}$, a contradiction.
Similar to the above $b_1 \notin N_{b_2c_2}$, we have $b_2 \notin N_{b_1c_1}$. Similarly, $c_2$ and the vertices in $V(\mathcal{F})\setminus (V(F_3))$ are not in $N_{b_1c_1}$.
Hence we have $b_1c_1a_1 \in \mathcal{F}$ or $b_1c_1a_2 \in \mathcal{F}$.
By symmetry, $b_1c_2a_1 \in \mathcal{F}$ or $b_1c_2a_2 \in \mathcal{F}$.
We claim that $b_1c_1a_1 \notin \mathcal{F}$ or $b_1c_2a_2 \notin \mathcal{F}$. Otherwise
let $e' \in {\{b_2,x\}\choose 1}\times{\{a_3,d_1,d_2\} \choose 2}$, if $e' \in \mathcal{F}$, then $\{e', xa_2b_2, b_1c_2a_2, b_1c_1a_1\}$ forms a copy of $P_4$ in $\mathcal{F}$, a contradiction. Hence $\left({\{b_2,x\} \choose 1}\times{\{a_3,d_1,d_2\} \choose 2}\right) \cap \mathcal{F} = \emptyset$.
Then Cover($b_2, d_1$) implies that $b_2d_1a_1 \in \mathcal{F}$ or $b_2d_1a_2 \in \mathcal{F}$. Then
$\{d_1d_2a_3, b_2d_1a_1, a_1c_1b_1, b_1c_2a_2\}$ or $\{d_2a_3d_1, d_1b_2a_2, a_2c_2b_1,b_1a_1c_1 \}$ forms a copy of $P_4$ in $\mathcal{F}$, a contradiction. Similarly, $b_1c_1a_2 \notin \mathcal{F}$ or $ b_1c_2a_1 \notin \mathcal{F}$.
Hence $b_1c_1a_1,b_1c_2a_1 \in \mathcal{F}$ or $b_1c_1a_2,b_1c_2a_2 \in \mathcal{F}$.

{\em Case 1.} $b_1c_1a_1,b_1c_2a_1 \in \mathcal{F}$.
Consider the property Cover$\{x, d_1\}$.
$b_1\notin N_{xd_1}$since otherwise $\{xb_1d_1,d_1d_2a_3,a_3a_1a_2,a_2c_1c_2\}$ forms a copy of $P_4$ in $\mathcal{F}$, a contradiction.
Similarly, all $b_2, c_1, c_2$ and the vertices not in $V(F'')$ are not in $N_{xd_1}$.
$a_1\notin N_{xd_1}$since otherwise  $\{d_1d_2a_3, d_1xa_1, a_1c_1b_1, c_1c_2a_2\}$ forms a copy of $P_4$ in $\mathcal{F}$, a contradiction.
$a_2\notin N_{xd_1}$since otherwise  $\{d_1d_2a_3, d_1xa_2, a_2c_1c_2, b_1c_2a_1\}$ forms a copy of $P_4$ in $\mathcal{F}$, a contradiction.
$a_3\notin N_{xd_1}$since otherwise  $\{xd_1a_3, xb_2a_2, a_2c_1c_2, b_1c_2a_1\}$ forms a copy of $P_4$ in $\mathcal{F}$, a contradiction.
Thus $\{x, d_1\}$ is not covered in any edge of $\mathcal{F}$, a contradiction.

{\em Case 2.} $b_1c_1a_2,b_1c_2a_2 \in \mathcal{F}$.
Consider the property Cover$\{x, d_1\}$.
The same as the above in Case 1, all $b_1, b_2, c_1, c_2$ and other vertices not in $V(F'')$ are not in $N_{xd_1}$.
$a_1\notin N_{xd_1}$since otherwise $\{d_1d_2a_3, d_1xa_1, a_1b_1b_2, b_1c_1a_2\}$ forms a copy of $P_4$ in $\mathcal{F}$, a contradiction.
$a_2\notin N_{xd_1}$since otherwise $\{d_1d_2a_3, d_1xa_2, a_2c_1b_1, b_1b_2a_1\}$ forms a copy of $P_4$ in $\mathcal{F}$, a contradiction.
Suppose $a_3\in N_{xd_1}$.
Now consider the property Cover$\{c_1,d_2\}$.
$b_1\notin N_{c_1d_2}$since otherwise  $\{xb_2a_2, a_1a_2a_3,a_3d_1d_2,c_1d_2b_1\}$ forms a copy of $P_4$ in $\mathcal{F}$, a contradiction.
$b_2\notin N_{c_1d_2}$since otherwise $\{b_1c_2a_2, a_1a_2a_3,a_3d_1d_2,c_1d_2b_2\}$ forms a copy of $P_4$ in $\mathcal{F}$, a contradiction.
$d_1\notin N_{c_1d_2}$since otherwise $\{b_1b_2a_1,a_1a_3a_2,a_2c_2c_1,c_1d_2d_1\}$ forms a copy of $P_4$ in $\mathcal{F}$, a contradiction.
Similarly, $c_2\notin N_{c_1d_2}$. Then $N_{c_1d_2}\subseteq\{a_1,a_2,a_3\}$.
If $a_1\notin N_{c_1d_2}$, then $\{xd_1a_3, xb_2a_2, a_2c_1c_2, c_1d_2a_1\}$ forms a copy of $P_4$ in $\mathcal{F}$, a contradiction.
If $a_2\notin N_{c_1d_2}$, then $\{a_1b_1b_2, xb_2a_2, a_2c_1d_2, d_2d_1a_3\}$ forms a copy of $P_4$ in $\mathcal{F}$, a contradiction.
If $a_3\notin N_{c_1d_2}$, then $\{c_1d_2a_3, xd_1a_3, xb_2a_2, a_1b_1b_2\}$ forms a copy of $P_4$ in $\mathcal{F}$, a contradiction.
Thus $\{x, d_1\}$ is not covered in any edge of $\mathcal{F}$, a contradiction.

Hence for every  $A \in \{ xb_ia_2,xb_ia_3, xc_ia_1,xc_ia_3, xd_ia_1,xd_ia_2:i\in [2]\}$, we have $A \notin \mathcal{F}$.
Let $y\in \{b_1,b_2, c_1, c_2, d_1, d_2\}$, consider the property Cover$(x,y)$.
We have proved that $N_{xy}\cap (\{b_1,b_2,c_1, c_2, d_1, d_2,$ $a_1,a_2,a_3\}\setminus \{y\})=\emptyset$.
$z\notin N_{xy}$ for some vertex not in $V(F'')\cup\{x\}$since otherwise $xyz$ connects to the endpoint of a linear path of 3 in $F''$ and we get a copy of $P_4$ in $\mathcal{F}$, a contradiction.
Hence $xb_ia_1, xc_ia_2, xd_ia_3 \in \mathcal{F}$ for every $i\in [2]$.
Consider the property Cover($b_2, c_2$).
We have proved that $b_1, c_1$ and vertex not in $V(F'')$ are not in $N_{b_2c_2}$.
Let $i=1,2$, we claim that $b_2c_2d_i \notin \mathcal{F}$; otherwise $\{b_1b_2a_1,a_1a_3a_2,a_2c_2c_1,c_1d_2d_1\}$ forms a copy of $P_4$ in $\mathcal{F}$, a contradiction.
$\{xb_1a_1, a_1a_2a_3, a_3d_1d_2, b_2c_2d_i\}$ forms a copy of $P_4$ in $\mathcal{F}$, a contradiction.
If $b_2c_2a_1 \in \mathcal{F}$, then
$\{d_1a_3x,xb_1a_1, a_1b_2c_2, c_2c_1a_2\}$ forms a copy of $P_4$ in $\mathcal{F}$, a contradiction.
If  $b_2c_2a_2 \in \mathcal{F}$, then $\{a_1b_1b_2, b_2c_2a_2,$ $a_2c_1x,xd_1a_3\}$ forms a copy of $P_4$ in $\mathcal{F}$, a contradiction.
If $b_2c_2a_3 \in \mathcal{F}$, then
$\{d_1d_2a_3, a_3b_2c_2, a_1b_2x, xc_1a_2\}$ forms a copy of $P_4$ in $\mathcal{F}$, a contradiction.
So $\{b_2,c_2\}$ is not covered by any edge of $\mathcal{F}$, a contradiction.
Hence $V(\mathcal{F})=V(F_3)$.

We deduce the value $\lambda(\mathcal{F})$.
Denote $B=\{b_1,b_2,c_1,c_2,d_1,d_2\}$, $E' = \{b_1b_2x,c_1c_2y,d_1d_2z:x\in B\setminus \{b_1,b_2\},y\in B\setminus \{c_1,c_2\},z\in B\setminus \{d_1,d_2\}\}$.
Since for every $e\in E'$, $e \notin \mathcal{F}$, then $\mathcal{F} \subseteq {V(\mathcal{F}) \choose 3}\setminus E':= \mathcal{G}$. Then
$ \lambda(\mathcal{F}) \le \lambda(\mathcal{G}) $. Let $\vec{x}$ be an optimum weighting of $\mathcal{G}$. By Lemma \ref{Equivalent}, we can assume that
$x_{a_1}=x_{a_2}=x_{a_3}=x$, $x_{b_1}=x_{b_2}=x_{c_1}=x_{c_2}=x_{d_1}=x_{d_2}=y$. So $3x+6y=1$.
Then $ \lambda(\mathcal{G}) = x^3+8y^3+18x^2y+45xy^2$. Let $L(x,y,\gamma)=x^3+8y^3+18x^2y+45xy^2-\gamma(3x+6y-1)$.
Then ${\partial L \over \partial x}={\partial L \over \partial y}={\partial L \over \partial \gamma}=0$ implies that
$y={\sqrt{873}-15 \over 162}$. So $$ \lambda(\mathcal{F}) \le \lambda(\mathcal{G}) \le 0.1035<7/ 64= \lambda(K_8^3),$$ which contradicts that $\lambda(\mathcal{F}) \ge \lambda(K_{8}^3)-0.005$.
This completes the proof.
\end{proof}

\section{Tur\'an numbers of the extensions}

If $L$ is a hypergraph on $[t]$, then a {\it blowup} of $L$ is a hypergraph $G$
whose vertex set can be partitioned into $V_1,\ldots, V_t$ such that
$E(G)=\bigcup_{e\in L} \prod_{i\in e} V_i$. Let $T_{t}^{r}(n)$ be the balanced blowup of $K_{t}^r$ on $n$ vertices, i.e., $V(T_{t}^{r}(n))=V_1 \cup V_2 \cup \dots \cup V_t$ such that $V_i \cap V_j=\emptyset$ for every $i \neq j $ and $|V_1| \le |V_2| \le \dots \le |V_t|\le |V_1|+1$, and $E(T_{t}^{r}(n))=
\{S\in \binom{[n]}{r}: \forall i\in [t], |S\cap V_i|\leq 1\}$.
Let $t_{t}^{r}(n)=|T_{t}^{r}(n)|$.
 The main result in this section is as follows.

 \begin{figure}[H]
\centering
\begin{minipage}{4cm}
\includegraphics[width=1.2\textwidth, height=1\textwidth]{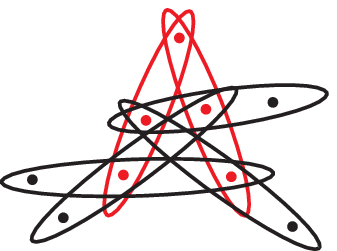}
\label{fig PP2}
\caption{ $H^{P_2}$}
\end{minipage}
\end{figure}

\begin{theo}\label{P_t-lagrangian}
Let $F\in \{P_3,P_4\}$. Then $ex(n,H^{F})= t_{|V(F)|-1}^{3}(n)$ for sufficiently large $n$. Moreover, if $n$ is sufficiently large and $G$ is an $H^{F}$-free $3$-graph on $[n]$ with $|G|=t_{|V(F)|-1}^{3}(n)$, then $G\cong T_{|V(F)|-1}^{3}(n)$.
\end{theo}

To prove the theorem, we need several results from  \cite{BIJ}. Similar results are obtained independently in  \cite{NY2}.
Let ${\mathcal{K}}_{p}^F $ denote the family of $r$-graphs $H$ that contains a set $C$ of $p$ vertices such that the subgraph of $H$ induced by $C$
contains a copy of $F$ and such that every pair in $C$ is covered in $H$.
Let $[m]_r=m\times (m-1)\times\dots \times (m-r+1)$.

\begin{defi}{\em (\cite{BIJ})
Let $m,r\ge 2$ be positive integers. Let $F$ be an $r$-graph that has at most $m+1$ vertices satisfying $\pi_{\lambda}(F)\le {[m]_r \over m^r} $. We say that ${\mathcal{K}}_{m+1}^F $ is {\em $m$-stable} if for every real $\varepsilon > 0$ there are a real $\delta > 0$ and an integer $n_1$ such that if $G$ is a ${\mathcal{K}}_{m+1}^F $-free $r$-graph with at least $n\ge n_1$ vertices and more than $({[m]_r \over m^r}-\delta){n \choose r} $ edges, then $G$ can be made $m$-partite by deleting at most $\varepsilon n$ vertices.}
\end{defi}

Given an $r$-graph $G$ and a real $\alpha$ with $0 < \alpha \le 1$, we say that $G$ is {\em $\alpha$-dense} if $G$ has minimum degree at least $\alpha {|V(G)|-1 \choose r-1}$. Let $i,j \in V(G)$, we say $i$ and $j$ are {\em nonadjacent} if $\{i,j\}$ is not contained in any edge of $G$. Given a set $U \subseteq V(G)$, we say $U$ is an {\em equivalence class} of $G$ if for every two vertices $u,v \in U$, $L_G(u)=L_G(v)$. Given two nonadjacent nonequivalent vertices $u,v \in V(G)$,  {\em symmetrizing} $u$ to $v$ refers to the operation of deleting all edges containing $u$ of $G$ and
adding all the edges $\{u\}\cup A, A\in L_G(v)$ to $G$.
We use the following algorithm from  \cite{BIJ}.

\begin{algorithm}\label{al1}{\em(Symmetrization and cleaning with threshold $\alpha$ \cite{BIJ})}
\newline
\noindent{\bf Input:} {\em An $r$-graph $G$.}
\newline
\noindent{\bf Output:} {\em An $r$-graph $G^*$.}
\newline
\noindent{\bf Initiation:} {\em  Let $G_0=H_0=G.$ Set $i=0$.}
\newline
\noindent{\bf Iteration:}
{\em For each vertex $u$ in $H_i$, let $A_i(u)$ denote the equivalence class that $u$ is in. If either $H_i$ is empty or $H_i$ contains no two nonadjacent nonequivalent vertices, then let $G^*=H_i$ and terminate. Otherwise let $u,v$ be two nonadjacent nonequivalent vertices in $H_i$, where $d_{H_i}(u) \ge d_{H_i}(v)$. We symmetrize each vertex in $A_i(v)$ to $u$. Let $G_{i+1}$ denote the resulting graph.
If $G_{i+1}$ is $\alpha$-dense, then let $H_{i+1}=G_{i+1}$. Otherwise let $L=G_{i+1}$ and repeat the following: let $z$ be any vertex of minimum degree in $L$. Redefine $L=L-z$ unless in forming $G_{i+1}$ from $H_i$ we symmetrized the equivalence class of some vertex $v$ in $H_i$ to some vertex in the equivalence class of $z$ in $H_i$.
In that case, we redefine $L=L-v$ instead. Repeat the process until $L$ becomes either $\alpha$-dense or empty. Let $H_{i+1}=L$.
We call the process of forming $H_{i+1}$ from $G_{i+1}$ ``cleaning". Let $Z_{i+1}$ denote the set of vertices removed, so $H_{i+1}=G_{i+1}-Z_{i+1}$. By our definition, if $H_{i+1}$ is nonempty then it is $\alpha$-dense.
}
\end{algorithm}

\begin{theo}{\em (\cite{BIJ})}\label{theo6}
Let $m,r\ge 2$ be positive integers. Let $F$ be an $r$-graph that has at most $m$ vertices or has $m+1$ vertices one of which has degree $1$. There exists a real $\gamma_0=\gamma_0(m,r)>0$ such that for every positive real $\gamma <\gamma_0$, there exist a real $\delta > 0$ and an integer $n_0$ such that the following is true for all $n\ge n_0$. Let $G$ be an ${\mathcal{K}}_{m+1}^F$-free $r$-graph on $[n]$ with more than $({[m]_r \over m^r}-\delta){n \choose r} $ edges. Let $G^*$ be the final $r$-graph produced by Algorithm \ref{al1} with threshold ${[m]_r \over m^r}-\gamma$. Then $|V(G^*)|\ge (1-\gamma)n$ and $G^*$ is $({[m]_r \over m^r}-\gamma)$-dense. Furthermore, if there is a set $W \subseteq V(G^*)$ with $|W|\ge (1-\gamma_0)|V(G^*)|$ such that $W$ is the union of a collection of at most $m$ equivalence classes of $G^*$, then $G[W]$ is $m$-partite. \qed
\end{theo}

\begin{theo}{\rm (\cite{BIJ})}\label{BIJ-main}
Let $m,r\ge 2$ be positive integers. Let $F$ be an $r$-graph that either has at most $m$ vertices or has $m+1$ vertices one of which has degree $1$. Suppose either $\pi_{\lambda}(F)< {[m]_r \over m^r} $ or $\pi_{\lambda}(F)={[m]_r \over m^r} $ and ${\mathcal{K}}_{m+1}^F $ is $m$-stable. Then there exists a positive integer $n_2$ such that for all $n \ge n_2$ we have $ex(n,H_{m+1}^{F})= t_{m}^{r}(n)$ and the unique extremal $r$-graph is $T_{m}^{r}(n)$.
\end{theo}

The following proposition follows immediately from the definition and is implicit in many papers (see  \cite{Keevash} for instance).
\begin{prop} \label{lag-bound}
Let $r\geq 2$. Let $L$ be an $r$-graph and $G$ be a blowup of $L$. Suppose $|V(G)|=n$. Then $|G|\leq \lambda(L) n^r$. \qed
\end{prop}

Theorem \ref{P_t-lagrangian} follows from the following Lemma, which is implicit in  \cite{BIJ} and  \cite{NY2}.
\begin{lemma}\label{stability}
Let $m,r\ge 2$ be positive integers. Let $F$ be an $r$-graph that has at most $m$ vertices or has $m+1$ vertices of which $r-1$ vertices of an edge has degree $1$ and $\pi_{\lambda}(F)\le {[m]_r \over m^r} $.
Suppose there is a constant $c>0$ such that
for every $F$-free and $K_{m}^r$-free $r$-graph $L$, $\lambda(L)\le \lambda(K_{m}^{r})-\epsilon$ holds. Then
${\mathcal{K}}_{m+1}^F $ is $m$-stable, consequently
$ex(n,H_{m+1}^{F})= t_{m}^{r}(n)$ and the unique extremal $r$-graph is $T_{m}^{r}(n)$.
\end{lemma}
\begin{proof} Let $\varepsilon >0$ be given. Let $\delta, n_0$ be the constants guaranteed by Theorem \ref{theo6}. We can assume that $\delta$ is small enough and $n_0$ is large enough. Let $\gamma >0$ satisfy $\gamma <\varepsilon$ and $ \delta+r\gamma < \epsilon$. Let $G$ be a ${\mathcal{K}}_{m+1}^{F}$-free $r$-graph on $n>n_0$ vertices with more than $({[m]_r \over m^r}-\delta){n \choose r} $ edges. Let $G^*$ be the final $r$-graph produced by applying Algorithm \ref{al1} to $G$ with threshold
${[m]_r \over m^r}-\gamma$. By Algorithm \ref{al1}, if $S$ consists of one vertex from each equivalence class of $G^*$, then $G^*[S]$ covers pairs and $G^*$ is a blowup of $G^*[S]$.

First, suppose that $|S| \ge m+1$. If $F \subseteq G^*[S]$, then since $G^*[S]$ covers pairs we can find a member of ${\mathcal{K}}_{m+1}^{F}$ in $G^*[S]$ by using any $(m+1)$-set that contains a copy of $F$ as the core, contradicting $G^*$ being ${\mathcal{K}}_{m+1}^{F}$-free.
So $G^*[S]$ is $F$-free.  We claim that $G^*[S]$ is $K^r_m$-free. Otherwise suppose $G^*[S]$ contains a copy of $K^r_m$.
When $|V(F)|=m$, $K^r_m$ contains a copy of $F$ clearly. So suppose that $|V(F)|=m+1$ and $F$ has $r-1$ vertices of one edge of degree $1$.
Let $e=\{v_1,\dots,v_r\}\in F$ with $d_F(v_1)=\dots=d_F(v_{r-1})=1$. Let $u_1\in S\setminus V(K^r_m)$ since $|S|\geq m+1$, and let $u_2\in V(K^r_m)$, since $G^*[S]$ covers pairs, there is an edge containing $\{u_1,u_2\}$ in $G^*[S]$, denote as $\{u_1,\dots,u_r\}$.
Assume that $V(F)=\{v_1,\dots,v_{m+1}\}$.
We define an injective function $f$ from $V(F)$ to $S$ with $f(v_i)=u_i$ for every $i\in [m+1]$, where $u_{r+1},\dots,u_{m+1}$ are arbitrary $m+1-r$ vertices in $V(K^r_m)\setminus \{u_2,\dots,u_r\}$.
It's clear that $f$ preserves edges and hence $G^*[S]$ contains a copy of $F$, a contradiction.
Thus, by our assumption, $\lambda(G^*[S])\leq \frac{1}{r!} \frac{[m]_r}{m^r}-\epsilon$. By Proposition \ref{lag-bound}, we have
\begin{equation} \label{G^*-upper}
|G^*|\leq \lambda(G^*[S])n^r\leq \left(\frac{1}{r!}\frac{[m]_r}{m^r}-\epsilon\right)n^r< \left(\frac{[m]_r}{m^r}-\epsilon\right)\frac{n^r}{r!}.
\end{equation}
Now, during the process of obtaining $G^*$ from $G$, symmetrization never decreases the number of edges.
Since at most $\gamma n$ vertices are deleted in the process (see Theorem \ref{theo6}),
$$|G^*|>|G|-\gamma n {n-1 \choose r-1} \ge \left({[m]_r \over m^r}-\delta-r\gamma \right){n \choose r} > \left({[m]_r \over m^r}-\epsilon\right){n^r \over r!},$$
contradicting \eqref{G^*-upper}. So $|S|\le m$. Hence, $W=V(G^*)$ is the union of at most $m$ equivalence classes of $G^*$. By Theorem \ref{theo6},
$|W|\geq (1-\gamma) n$ and $G[W]$ is $m$-partite. Hence, $G$ can be made $m$-partite by deleting at most $\gamma n<\varepsilon n$ vertices. Thus,
$\mathcal{K}_{m+1}^F$ is $m$-stable. Then the result holds by Theorem \ref{BIJ-main}.
\end{proof}




\section{Remark}

If one could show that left-compressing a dense $P_t$-free $3$-graph will result in  a $P_t$-free $3$-graph, then it would  not be  hard to show that $P_t$ is perfect. It seems to be hard  for general $t$.  All known cases regarding Lagrangian densities are listed in Section 1 (to the best of our knowledge). For $3$-uniform graphs spanned by 3  edges,
there  is still one remaining unsolved cases:  $K_4^{3-}$ (\{123, 124, 134\}). Since  the Lagrangian density and the Tur\'an density of $K_4^{3-}$ equal, it would be very interesting in determining the Tur\'an density of $K_4^{3-}$  by its Lagrangian density.

\bigskip

{\bf Acknowledgement}
We thank Tao Jiang for helpful discussions.

\end{document}